\documentclass[12pt]{article}

\usepackage{amsmath}
\usepackage{amsthm}
\usepackage{amssymb}

\usepackage{cite}
\usepackage[multiple]{footmisc}

\usepackage{enumitem}
\setlist[enumerate]{label={\textnormal{(\roman*)}}}

\newtheorem{theorem}{Theorem}[]
\newtheorem{proposition}[theorem]{Proposition}
\newtheorem{lemma}[theorem]{Lemma}

\theoremstyle{definition}
\newtheorem{conjecture}[theorem]{Conjecture}

\renewcommand{\tt}[1]{\texttt{#1}}
\DeclareMathOperator{\Fact}{Fact}
\DeclareMathOperator{\Stab}{Stab}
\DeclareMathOperator{\RT}{RT}

\title{The Number of Threshold Words on $n$ Letters Grows Exponentially for Every $n\geq 27$}
\author{James~D.~Currie\footnote{The work of James D. Currie is supported by the Natural Sciences and Engineering Research Council of Canada (NSERC), [funding reference number 2017-03901].}, Lucas~Mol, and Narad~Rampersad\footnote{The work of Narad Rampersad is supported by the Natural Sciences and Engineering Research Council of Canada (NSERC), [funding reference number 2019-04111].}\\
\small Department of Mathematics and Statistics\\
\small The University of Winnipeg\\
\small 515 Portage Ave.\\
\small Winnipeg, MB, Canada\\
\small R3B 2E9\\
\small \{j.currie, l.mol, n.rampersad\}@uwinnipeg.ca}
\date{}

\begin{document}

\maketitle

\begin{abstract}
\noindent
For every $n\geq 27$, we show that the number of $n/(n-1)^+$-free words (i.e., threshold words) of length $k$ on $n$ letters grows exponentially in $k$.  This settles all but finitely many cases of a conjecture of Ochem.


\end{abstract}

\section{Introduction}

Throughout, we use standard definitions and notations from combinatorics on words (see~\cite{LothaireAlgebraic}).  A \emph{square} is a word of the form $xx$, where $x$ is a nonempty word.  A \emph{cube} is a word of the form $xxx$, where $x$ is a nonempty word. An \emph{overlap} is a word of the form $axaxa$, where $a$ is a letter and $x$ is a (possibly empty) word.  The study of words goes back to Thue, who demonstrated the existence of an infinite overlap-free word over a binary alphabet, and an infinite square-free word over a ternary alphabet  (see~\cite{Berstel1995}).

A language is a set of finite words over some alphabet $A$.  The \emph{combinatorial complexity} of a language $L$ is the sequence $C_L:\mathbb{N}\rightarrow\mathbb{N}$, where $C_L(k)$ is defined as the number of words in $L$ of length $k$.  We say that a language $L$ \emph{grows} exactly as the sequence $C_L(k)$ grows, be it exponentially, polynomially, etc.
Since the work of Brandenburg~\cite{Brandenburg1983}, the study of the growth of languages has been a central theme in combinatorics on words.  Given a language $L$, a key question is whether it grows exponentially (fast), or subexponentially (slow).  Brandenburg~\cite{Brandenburg1983} demonstrated that both the language of cube-free words over a binary alphabet, and the language of square-free words over a ternary alphabet, grow exponentially.  On the other hand, Restivo and Salemi~\cite{RestivoSalemi1985} demonstrated that the language of overlap-free binary words grows only polynomially.

Squares, cubes, and overlaps are all examples of \emph{repetitions} in words, and can be considered in the same general framework.  Let $w=w_1w_2\cdots w_k$ be a finite word, where the $w_i$'s are letters.  A positive integer $p$ is a \emph{period} of $w$ if $w_{i+p}=w_i$ for all $1\leq i\leq k-p$.  In this case, we say that $|w|/p$ is an \emph{exponent} of $w$, and the largest such number is called \emph{the} exponent of $w$.  For a real number $r>1$, a finite or infinite word $w$ is called $r$-free ($r^+$-free) if $w$ contains no finite factors of exponent greater than or equal to $r$ (strictly greater than $r$, respectively).

Throughout, for every positive integer $n$, let $A_n$ denote the $n$-letter alphabet $\{\tt{1},\tt{2},\dots,\tt{n}\}$.  For every $n\geq 2$, the \emph{repetition threshold} for $n$ letters, denoted $\RT(n)$, is defined by
\[
\RT(n)=\inf\{r>1\colon\ \text{there is an infinite $r^+$-free word over $A_n$}\}.
\]
Essentially, the repetition threshold describes the border between avoidable and unavoidable repetitions in words over an alphabet of $n$ letters.  The repetition threshold was first defined by Dejean~\cite{Dejean1972}.  Her 1972 conjecture on the values of $\RT(n)$ has now been confirmed through the work of many authors~\cite{Dejean1972,Pansiot1984,MoulinOllagnier1992,MohammadNooriCurrie2007,Carpi2007,CurrieRampersad2009,CurrieRampersad2009Again,CurrieRampersad2011,Rao2011}:
\[
\RT(n)=\begin{cases}
2, &\text{ if $n=2$};\\
7/4, & \text{ if $n=3$};\\
7/5, & \text{ if $n=4$};\\
n/(n-1), & \text{ if $n\geq 5$.}
\end{cases}
\]
The last cases of Dejean's conjecture were confirmed in 2011 by the first and third authors~\cite{CurrieRampersad2011}, and independently by Rao~\cite{Rao2011}.  However, probably the most important contribution was made by Carpi~\cite{Carpi2007}, who confirmed the conjecture in all but finitely many cases.

In this short note, we are concerned with the growth rate of the language of \emph{threshold words} over $A_n$.  For every $n\geq 2$, let $T_n$ denote the language of all $\RT(n)^+$-free words over $A_n$.  We call $T_n$ the \emph{threshold language} of order $n$, and we call its members \emph{threshold words} of order $n$.  Threshold words are also called \emph{Dejean words} by some authors.  For every $n\geq 2$, the threshold language $T_n$ is the minimally repetitive infinite language over $A_n$.

The threshold language $T_2$ is exactly the language of overlap-free words over $A_2$, which is known to grow only polynomially~\cite{RestivoSalemi1985}.\footnote{Currently the best known bounds on $C_{T_2}(k)$ are due to Jungers et al.~\cite{JungersProtasovBlondel2009}.}\footnote{The threshold between polynomial and exponential growth for repetition-free binary words is known to be $7/3$~\cite{KarhumakiShallit2004}.  That is, the language of $7/3$-free words over $A_2$ grows polynomially, while the language of $7/3^+$-free words over $A_2$ grows exponentially.}  However, Ochem made the following conjecture about the growth of threshold languages of all other orders.

\begin{conjecture}[Ochem~\cite{Ochem2006}]
\label{OchemConjecture}
For every $n\geq 3$, the language $T_n$ of threshold words of order $n$ grows exponentially.
\end{conjecture}

Conjecture~\ref{OchemConjecture} has been confirmed for $n\in\{3,4\}$ by Ochem~\cite{Ochem2006}, for $n\in\{5,6,\dots,10\}$ by Kolpakov and Rao~\cite{KolpakovRao2011}, and for all odd $n$ less than or equal to $101$ by Tunev and Shur~\cite{TunevShur2012}.  In this note, we confirm Conjecture~\ref{OchemConjecture} for every $n\geq 27$.

\begin{theorem}\label{main}
For every $n\geq 27$, the language $T_n$ of threshold words of order $n$ grows exponentially.
\end{theorem}

The layout of the remainder of the note is as follows.  In Section~\ref{Review}, we summarize the work of Carpi~\cite{Carpi2007} in confirming all but finitely many cases of Dejean's conjecture.  In Section~\ref{Construction}, we establish Theorem~\ref{main} with constructions that rely heavily on the work of Carpi.  
We conclude with a discussion of problems related to the rate of growth of threshold languages.

\section{Carpi's reduction to \boldmath{$\psi_n$}-kernel repetitions}\label{Review}

In this section, let $n\geq 2$ be a fixed integer.  Pansiot~\cite{Pansiot1984} was first to observe that if a word over the alphabet $A_n$ is $(n-1)/(n-2)$-free, then it can be encoded by a word over the binary alphabet $B=\{\tt{0},\tt{1}\}$.  For consistency, we use the notation of Carpi~\cite{Carpi2007} to describe this encoding.  Let $\mathbb{S}_n$ denote the symmetric group on $A_n$, and define the morphism $\varphi_n:B^*\rightarrow \mathbb{S}_n$ by
\begin{align*}
\varphi_n(\tt{0})&=\begin{pmatrix}
\tt{1} & \tt{2} & \cdots & \tt{n-1}
\end{pmatrix}; \text{ and}\\
\varphi_n(\tt{1})&=\begin{pmatrix}
\tt{1} & \tt{2} & \cdots & \tt{n}
\end{pmatrix}.
\end{align*}
Now define the map $\gamma_n:B^*\rightarrow A_n^*$ by
\[
\gamma_n(b_1b_2\cdots b_k)=a_1a_2\cdots a_k,
\]
where 
\[
a_i\varphi_n(b_1b_2\cdots b_i)=\tt{1}
\]
for all $1\leq i\leq k$.  To be precise, Pansiot proved that if a word $\alpha\in A_n^*$ is $(n-1)/(n-2)$-free, then $\alpha$ can be obtained from a word of the form $\gamma_n(u)$, where $u\in B^*$, by renaming the letters.

Let $u\in B^*$, and let $\alpha=\gamma_n(u)$.  Pansiot showed that if $\alpha$ has a factor of exponent greater than $n/(n-1)$, then either the word $\alpha$ itself contains a \emph{short repetition}, or the binary word $u$ contains a \emph{kernel repetition} (see~\cite{Pansiot1984} for details).  Carpi reformulated this statement so that both types of forbidden factors appear in the binary word $u$.  Let $k\in\{1,2,\dots,n-1\}$, and let $v\in B^+$.  Then $v$ is called a \emph{$k$-stabilizing word} (of order $n$) if $\varphi_n(v)$ fixes the points $\tt{1},\tt{2},\dots,\tt{k}$.  Let $\Stab_n(k)$ denote the set of $k$-stabilizing words of order $n$.  The word $v$ is called a \emph{kernel repetition} (of order $n$) if it has period $p$ and a factor $v'$ of length $p$ such that $v'\in \ker(\varphi_n)$ and $|v|>\frac{np}{n-1}-(n-1)$.
Carpi's reformulation of Pansiot's result is the following.

\begin{proposition}[Carpi~{\cite[Proposition~3.2]{Carpi2007}}]
\label{CarpiReform}
Let $u\in B^*$.  If a factor of $\gamma_n(u)$ has exponent larger than $n/(n-1)$, then $u$ has a factor $v$ satisfying one of the following conditions:
\begin{enumerate}
    \item $v\in \Stab_n(k)$ and $0<|v|<k(n-1)$ for some $1\leq k\leq n-1$; or
    \item $v$ is a kernel repetition of order $n$.
\end{enumerate}
\end{proposition}

Now assume that $n\geq 9$, and define $m=\lfloor (n-3)/6\rfloor$ and $\ell=\lfloor n/2\rfloor$.  Carpi~\cite{Carpi2007} defines an $(n-1)(\ell+1)$-uniform morphism $f_n:A_m^*\rightarrow B^*$ with the following extraordinary property.

\begin{proposition}[Carpi~{\cite[Proposition~7.3]{Carpi2007}}]
\label{CarpiShort}
Suppose that $n\geq 27$, and let $w\in A_m^*$. Then for every $k\in\{1,2,\dots,n-1\}$, the word $f_n(w)$ contains no $k$-stabilizing word of length smaller than $k(n-1)$.
\end{proposition}

We note that Proposition~\ref{CarpiShort} was proven by Carpi~\cite{Carpi2007} in the case that $n\geq 30$ in a computation-free manner.  The improvement to $n\geq 27$ stated here was achieved later by the first and third authors~\cite{CurrieRampersad2009Again}, using lemmas of Carpi~\cite{Carpi2007} along with a significant computer check.

Proposition~\ref{CarpiShort} says that for \emph{every} word $w\in A_m^*$, no factor of $f_n(w)$ satisfies condition (i) of Proposition~\ref{CarpiReform}. 
Thus, we need only worry about factors satisfying condition (ii) of Proposition~\ref{CarpiReform}, i.e., kernel repetitions.  To this end, define the morphism $\psi_n:A_m^*\rightarrow \mathbb{S}_n$ by $\psi_n(v)=\varphi_n(f_n(v))$ for all $v\in A_m^*$.  A word $v\in A_m^*$ is called a \emph{$\psi_n$-kernel repetition} if it has a period $q$ and a factor $v'$ of length $q$ such that $v'\in \ker(\psi_n)$ and $(n-1)(|v|+1)\geq nq-3$.  Carpi established the following result.

\begin{proposition}[Carpi~{\cite[Proposition~8.2]{Carpi2007}}]
\label{CarpiKernel}
Let $w\in A_m^*$.  If a factor of $f_n(w)$ is a kernel repetition, then a factor of $w$ is a $\psi_n$-kernel repetition.
\end{proposition}

In other words, if $w\in A_m^*$ contains no $\psi_n$-kernel repetitions, then no factor of $f_n(w)$ satisfies condition (ii) of Proposition~\ref{CarpiReform}.  Altogether, we have the following theorem, which we state formally for ease of reference.

\begin{theorem}\label{CarpiMain}
Suppose that $n\geq 27$.  If $w\in A_m^*$ contains no $\psi_n$-kernel repetitions, then $\gamma_n(f_n(w))$ is $\RT(n)^+$-free.
\end{theorem}

Finally, we note that the morphism $f_n$ is defined in such a way that  the kernel of $\psi_n$ has a very simple structure.

\begin{lemma}[Carpi~{\cite[Lemma~9.1]{Carpi2007}}]
\label{Carpi4}
If $v\in A_m^*$, then $v\in \ker(\psi_n)$ if and only if $4$ divides $|v|_a$ for every letter $a\in A_m$.
\end{lemma}

\section{Constructing exponentially many threshold words}\label{Construction}

In this section, let $n\geq 27$ be a fixed integer, and let $m=\lfloor (n-3)/6\rfloor$ and $\ell=\lfloor n/2\rfloor$, as in the previous section.  Since $n\geq 27$, we have $m\geq 4$.  In order to prove that the threshold language $T_n$ grows exponentially, we construct an exponentially growing language $Z_m\subseteq A_m^*$ of words that contain no $\psi_n$-kernel repetitions.  If $n\geq 33$ (or equivalently, if $m\geq 5$), then we define $Z_m$ by modifying Carpi's construction of an infinite word $\alpha$ over $A_m$ that contains no $\psi_n$-kernel repetitions.  If $27\leq n\leq 32$ (or equivalently, if $m=4$), then we define a $3$-uniform substitution $g\colon A_4^*\rightarrow 2^{A_4^*}$, and let $Z_4$ be the set of all factors of words obtained by iterating $g$ on the letter $\tt{1}$.

\subsection*{Case I: $n\geq 33$}

We first recall the definition of $\alpha$, the infinite word over $A_m$ defined by Carpi~\cite{Carpi2007} that contains no $\psi_n$-kernel repetitions.  First of all, define $\beta=(b_i)_{i\geq 1}$, where
\[
b_i=\begin{cases}
\tt{1}, &\text{ if $i\equiv 1\pmod{3}$;}\\
\tt{2}, &\text{ if $i\equiv 2\pmod{3}$;}\\
b_{i/3}, &\text{ if $i\equiv 0\pmod{3}$.}
\end{cases}
\]
Now define $\alpha=(a_i)_{i\geq 1}$, where for all $i\geq 1$, we have
\[
a_i=\begin{cases}
\max\{a\in A_{m}\colon\ 4^{a-2} \text{ divides } i\}, & \text{ if $i$ is even};\\
b_{(i+1)/2}, & \text{ if $i$ is odd}.
\end{cases}
\]

Note that if $i\equiv 2$ (mod $4$), then $a_i=\tt{2}$.  Let $Z_m$ be the set of all finite words obtained from a prefix of $\alpha$ by exchanging any subset of these \tt{2}'s for \tt{1}'s.  To be precise, if $z=z_1z_2\cdots z_k$, then $z\in Z_m$ if and only if all of the following hold:
\begin{itemize}
\item $z_i\in\{\tt{1},\tt{2}\}$ if $i\equiv 2$ (mod $4$);
\item $z_i=\max\left\{a\in A_m\colon\ \text{$4^{a-2}$ divides $i$}\right\}$ if $i\equiv 0$ (mod $4$); and
\item $z_i=b_{(i+1)/2}$ if $i$ is odd. 
\end{itemize}
Note in particular that if $z=z_1z_2\cdots z_k$ is in $Z_m$, then $z_i\geq 3$ if and only if $i\equiv 0$ (mod $4$).

We claim that no word $z\in Z_m$ contains a $\psi_n$-kernel repetition.  The proof is essentially analogous to Carpi's proof that $\alpha$ contains no $\psi_n$-kernel repetitions.  We begin with a lemma about the lengths of factors in $Z_m$ that lie in $\ker(\psi_n)$.

\begin{lemma}[Adapted from Carpi~{\cite[Lemma~9.3]{Carpi2007}}]\label{Analogous}
Let $z\in Z_m$, and let $v$ be a factor of $z$.  If $v\in\ker(\psi_n)$, then $4^{m-1}$ divides $|v|$.
\end{lemma}

\begin{proof}
The statement is trivially true if $v=\varepsilon$, so assume $|v|>0$.  Set $|v|=4^bc$, where $4^b$ is the maximal power of $4$ dividing $|v|$.  Suppose, towards a contradiction, that $b\leq m-2$.  Since $v\in\ker(\psi_n)$, by Lemma~\ref{Carpi4}, we see that $4$ divides $|v|$, meaning $b\geq 1$.

Write $z=z_1z_2\cdots z_{|z|}$.  Then we have $v=z_iz_{i+1}\cdots z_{i+4^bc-1}$ for some $i\geq 1$.  By definition, for any $j\geq 1$, we have $z_j\geq b+2$ if and only if $4^b$ divides $j$.  (Since $b\geq 1$, we have $b+2\geq 3$, and hence $z_j\geq b+2$ implies $j\equiv 0$ (mod $4$).)  Thus, we have that the sum $\sum_{a=b+2}^{m}|v|_a$ is exactly the number of integers in the set $\{i,i+1,\dots,i+4^bc-1\}$ that are divisible by $4^b$, which is exactly $c$.  Since $v\in \ker(\psi_n)$, by Lemma~\ref{Carpi4}, we conclude that $4$ divides $c$, contradicting the maximality of $b$.
\end{proof}

Now, using Lemma~\ref{Analogous} in place of~\cite[Lemma~9.3]{Carpi2007}, a proof strictly analogous to that of~\cite[Proposition~9.4]{Carpi2007} gives the following.  The only tool in the proof that we have not covered here is~\cite[Lemma~9.2]{Carpi2007}, which is a short technical lemma about the repetitions in the word $\beta$, and which can be used without any modification.

\begin{proposition}\label{zeta}
Suppose that $n\geq 33$.  Then no word $z\in Z_m$ contains a $\psi_n$-kernel repetition.
\end{proposition}

\subsection*{Case II: $27\leq n\leq 32$}

Define a substitution $g:A_4^*\rightarrow 2^{A_4^*}$ by
\begin{align*}
g(\tt{1})&=\{\tt{112}\}\\
g(\tt{2})&=\{\tt{114}\}\\
g(\tt{3})&=\{\tt{113}\}\\
g(\tt{4})&=\{\tt{123},\tt{213}\}.
\end{align*}
We extend $g$ to $2^{A_4^*}$ by $g(W)=\bigcup_{w\in W}g(w)$, which allows us to iteratively apply $g$ to an initial word in $A_4^*$.  Let $Z_4=\Fact\{v\colon\ \text{$v\in g^n(\tt{1})$ for some $n\geq 1$}\}$, i.e., we have that $Z_4$ is the set of factors of all words obtained by iteratively applying $g$ to the initial word $\tt{1}$.
If a word $w\in A_4^*$ has period $p$ and the length $p$ prefix of $w$ is in $\ker(\psi)$, then we say that $p$ is a \emph{kernel period} of $w$.

\begin{proposition}\label{m4}
Suppose that $27\leq n\leq 32$.  Then no word in $Z_4$ contains a $\psi_n$-kernel repetition.
\end{proposition}

\begin{proof}
Suppose otherwise that the word $v_0\in Z_4$ is a $\psi_n$-kernel repetition.  Write $v_0=x_0y_0$, where $v_0$ has kernel period $|x_0|$.  Without loss of generality, we may assume that no extension of $v_0$ that lies in $Z_4$ has period $|x_0|$, i.e., that $v_0$ is a maximal repetition in $Z_4$.  From the definition of $\psi_n$-kernel repetition, we must have
\[
(n-1)(|v_0|+1)\geq n|x_0|-3,
\] 
or equivalently,
\[
|x_0|\leq (n-1)|y_0|+n+2.
\]
Since $n\leq 32$, we certainly have
\begin{align}\label{psikernelcondition}
|x_0|\leq 31|y_0|+34.
\end{align}
If $|y_0|\leq 3$, then we have $|x_0|\leq 127$, and hence $|v_0|\leq 130$.  We eliminate this possibility by exhaustive search, so we may assume that $|y_0|\geq 4$.

We can write $v_0=s_0v'_0p_0$ for some suffix $s_0$ of a word in $g(A_4)$,  some prefix $p_0$ of a word in $g(A_4)$, and some word $v'_0\in g(v_1)$, where $v_1\in Z_4$.  By inspection, we see that if $z$ is any factor of $Z_4$ of length $3$, and both $\pi_1z$ and $\pi_2z$ are prefixes of some word in $Z_4$, then $|\pi_1|\equiv |\pi_2|$ (mod $3$).  Since both $y_0$ and $x_0y_0=v_0$ are prefixes of $v_0$, and since $|y_0|> 3$, we conclude that $|x_0|$ is a multiple of $3$.

Recall that we have $v_0=s_0v'_0p_0$, where $v'_0\in g(v_1)$ for some word $v_1\in Z_4$.  Since $|s_0|\leq 2$ and $|p_0|\leq 2$, we have $v'_0\geq|v_0|-4\geq |x_0|$, and hence $v'_0$ has kernel period $|x_0|$.  Now write $v_1=x_1y_1$, where $3|x_1|=|x_0|$. Evidently, we have $3|y_1|+4\geq |y_0|$.  Note that $v_1$ has period $|x_1|$.  Further, since the frequency matrix of $g$ is invertible modulo $4$, we have $x_1\in \ker(\psi_n)$, and hence $|x_1|$ is a kernel period of $v_1$.  Since $v_0$ was a maximal repetition in $Z_4$, we see that $v_1$ is also maximal.

We may now repeat the process described above.  Eventually, for some $r\geq 1$, we reach a word $v_r\in Z_4$ that can be written $v_r=x_ry_r$, where $|x_r|$ is a kernel period of $v_r$, and $|y_r|\leq 3$.  For all $1\leq i\leq r$, one proves by induction that $|x_0|=3^i|x_i|$ and $|y_0|\leq 3^i|y_i|+4\sum_{j=0}^{i-1}3^j=3^i|y_i|+2(3^i-1)$.
Thus, from (\ref{psikernelcondition}), we obtain
\begin{align*}
3^i|x_i|\leq 31\left[3^i|y_i|+2(3^i-1)\right]+34
\end{align*}
for all $1\leq i\leq r$.
Dividing through by $3^i$, and then simplifying, we obtain
\begin{align}\label{Bound}
|x_i|\leq 31\left[|y_i|+2\right]-\frac{28}{3^i}\leq 31\left[|y_i|+2\right] 
\end{align}
for all $1\leq i\leq r$.  

Since $|y_r|\leq 3$, we obtain $|x_r|\leq 155$ from (\ref{Bound}).  By Lemma~\ref{Carpi4}, the kernel period $|x_r|$ of $v_r$ is a multiple of $4$, so in fact we have $|x_r|\leq 152$, and in turn $|v_r|\leq 155$.  By exhaustive search of all words in $Z_4$ of length at most $155$, we find that $v_r\in W$, where $W$ is a set containing exactly $200$ words.  Indeed, the set $W$ contains
\begin{itemize}
\item 160 words with kernel period 76 and length 77,
\item 36 words with kernel period 92 and length 93, and
\item 4 words with kernel period 112 and length 114.
\end{itemize}
For every $w\in W$, let
\[
E_w=\Fact\left(\{g(awb)\colon\ \text{$a,b\in A_4$, $awb\in Z_4$}\}\right).
\] 
Evidently, we have $v_{r-1}\in E_{v_r}$.  For every word $w\in W$, let $p_w$ denote the kernel period of $w$, and let $q_w$ denote the maximum length of a repetition with kernel period $3p_w$ across all words in $E_w$.  By exhaustive check, for every $w\in W$, we find $3p_w> 31\left[q_w-3p_w+2\right]$.  However, the word $v_{r-1}=x_{r-1}y_{r-1}$ must be in $E_{v_r}$, and by (\ref{Bound}), we have 
\[
3p_{v_r}=|x_{r-1}|\leq 31\left[|y_{r-1}|+2\right]\leq 31\left[q_{v_r}-3p_{v_r}+2\right].
\]  
This is a contradiction.  We conclude that the set $Z_4$ contains no $\psi_n$-kernel repetitions.
\end{proof}

We now proceed with the proof of our main result.

\begin{proof}[Proof of Theorem~\ref{main}]
First suppose that $n\geq 33$.  By Proposition~\ref{zeta}, no word $z\in Z_m$ contains a $\psi_n$-kernel repetition.  From the definition of $Z_m$, one easily proves that
\[
C_{Z_m}(k)=\Omega\left(2^{k/4}\right).
\]
By Theorem~\ref{CarpiMain}, for every word $z\in Z_m$, the word $\gamma_n(f_n(z))$ is in the threshold language $T_n$ of order $n$.  Moreover, the maps $\gamma_n$ and $f_n$ are injective, and $|\gamma_n(f_n(z))|/|z|=(n-1)(\ell+1)$, since $f_n$ is $(n-1)(\ell+1)$-uniform and $\gamma_n$ preserves length.  It follows that 
\[
C_{T_n}(k)=\Omega\left(2^{k/4(n-1)(\ell+1)}\right).
\]  
Since $n$, and hence $\ell$, are fixed, the quantity $(n-1)(\ell+1)$ is a constant, and we conclude that the language $T_n$ grows exponentially.

Suppose now that $27\leq n\leq 32$.  By Proposition~\ref{m4}, no word $z\in Z_4$ contains a $\psi_n$-kernel repetition.  Since $|g^4(a)|\geq 4$ for all $a\in A_4$, we have
\[
C_{Z_4}(k)=\Omega\left(4^{k/81}\right).
\]
By the same argument as above, we see that 
\[
C_{T_n}(k)=\Omega\left(4^{k/81(n-1)(\ell+1)}\right),
\]
and we conclude that the language $T_n$ grows exponentially.
\end{proof}

\section{Conclusion}

Conjecture~\ref{OchemConjecture} has now been established for all $n\not\in\{12,14,\dots,26\}$.  We remark that different techniques than those presented here will be needed to establish Conjecture~\ref{OchemConjecture} in all but one of these remaining cases.  (It appears that the techniques presented here could potentially be used for $n=22$, but we do not pursue this isolated case.)  For example, let $n=26$.  Then we have $m=3$.  By computer search, for every letter $a\in A_m$, the word $f_n(a\tt{3})$ contains a $15$-stabilizing word of length $350$, which is less than $15(n-1)=375$.  By another computer search, the longest word on $\{\tt{1},\tt{2}\}$ avoiding $\psi_{n}$-kernel repetitions has length $15$.  So there are only finitely many words in $A_m^*$ that avoid both $\psi_n$-kernel repetitions and the forbidden stabilizing words.  Similar arguments lead to the same conclusion for all $n\in\{12,14,16,18,20,24\}$.

For a language $L$, the value
$\alpha(L)=\limsup_{k\rightarrow\infty}(C_L(k))^{1/k}$
is called the \emph{growth rate} of $L$.  If $L$ is factorial (i.e., closed under taking factors), then by an application of Fekete's Lemma, we can safely replace $\limsup$ by $\lim$ in this definition.  If $\alpha(L)>1$, then the language $L$ grows exponentially, and in this case, $\alpha(L)$ is a good description of how quickly the language grows.

For all $n\geq 33$, we have established that $\alpha(T_n)\geq 2^{1/4(n-1)\left(\ell+1\right)}$.  However, this lower bound tends to $1$ as $n$ tends to infinity, and this seems far from best possible.  Indeed, Shur and Gorbunova proposed the following conjecture concerning the asymptotic behaviour of $\alpha(T_n)$.

\begin{conjecture}[Shur and Gorbunova~\cite{ShurGorbunova2010}]
The sequence $\{\alpha(T_n)\}$ of the growth rates of threshold languages converges to a limit $\hat{\alpha}\approx 1.242$ as $n$ tends to infinity.
\end{conjecture}

A wide variety of evidence supports this conjecture -- we refer the reader to~\cite{ShurGorbunova2010,GorbunovaShur2011,Shur2014,Shur2012} for details.  For a fixed $n$, there are efficient methods for determining upper bounds on $\alpha(T_n)$ which appear to be rather sharp, even for relatively large values of $n$ (see~\cite{ShurGorbunova2010}, for example).  Establishing a sharp lower bound on $\alpha(T_n)$ appears to be a more difficult problem.  We note that a good lower bound on $\alpha(T_3)$ is given by Kolpakov~\cite{Kolpakov2007} using a method that requires some significant computation.  For all $n\in\{5,6,\dots,10\}$, Kolpakov and Rao~\cite{KolpakovRao2011} give lower bounds for $\alpha(T_n)$ using a similar method.  They were then able to estimate the value of $\alpha(T_n)$ with precision 0.005 using upper bounds obtained by the method of Shur and Gorbunova~\cite{ShurGorbunova2010}.  

Thus, in addition to resolving the finitely many remaining cases of Conjecture~\ref{OchemConjecture}, improving our lower bound for $\alpha(T_n)$ when $n\geq 27$ remains a significant open problem.

\end{document}